\documentclass[preprint,12pt]{elsarticle}

\usepackage{amsmath}
\usepackage{amssymb}
\usepackage{amsthm}
\usepackage{setspace}

\newtheorem*{theorem}{Theorem}
\newtheorem{lemma}{Lemma}

\DeclareMathOperator{\Aut}{Aut}

\journal{arXiv.org}

\begin{document}

\begin{frontmatter}

\title{An elementary proof of a power series identity for the weighted sum of all finite abelian $p$-groups}

\author{Pritam Majumder}

\address{Indian Institute of Technology Kanpur, India}

\doublespacing
\begin{abstract}

Using combinatorial techniques, we prove that the weighted sum of the inverse number of automorphisms of all finite abelian $p$-groups $\sum_G |G|^{-u} \left|\Aut (G)\right|^{-1}$ is equal to $\prod_{j=u+1}^\infty\left(1-1/p^j\right)^{-1}$, where $u$ is a non-negative integer. This result was originally obtained by H. Cohen and H. W. Lenstra, Jr. In this paper we give a new elementary proof of their result.

\end{abstract}

\end{frontmatter}

\section{Introduction}
\label{S:1}
Let $\mathcal{G}_p$ denote the set of all finite abelian $p$-groups, for a prime number $p$. We will give an elementary combinatorial proof of the following theorem:

\begin{theorem}
For a prime $p$ and a non-negative integer $u$, the following holds
$$\sum_{G\in\mathcal{G}_p}\frac{1}{|G|^u\cdot \left|\Aut (G)\right|}=\prod_{j=u+1}^\infty\left(1-\frac{1}{p^j}\right)^{-1}.$$
\end{theorem} 

The above result was obtained by Cohen and Lenstra in their famous paper \cite{cohen}. Their approach is more complicated but has the advantage that it generalizes naturally to finite modules over the rings of integers of number fields. In the special case of $u=0$, a nice combinatorial proof was given by Hall, which can be found in \cite{hall}. Our proof is a generalization of the proof of Hall.

\section{Proof of the Theorem}
For $m\geq 0$, let $a_m$ be the number of partitions of $m$ with each part at least $u+1$ and for $i,j\geq 0$, let $b_{i,j}$ be the number of partitions of $i$ with greatest part exactly equal to $j$. 

\begin{lemma}  
For all $m\geq 0$, the following holds
$$a_m=\sum_{i+uj=m}b_{i,j}.$$
\end{lemma}

\begin{proof}
We will give a bijection argument. Note that, the number of partitions of $i$ with greatest part $j$, where $i+uj=m$, is equal to the number of partitions of $i+uj=m$ with greatest part $j$ occurring at least $u+1$ times. Hence $\sum_{i+uj=m}b_{i,j}$ is equal to the number of partitions of $m$ with greatest part occurring at least $u+1$ times. Now, a partition of $m$ has greatest part occurring at least $u+1$ times if and only if it's conjugate partition has each part at least $u+1$. This gives a bijection between the partitions of $m$ with greatest part occurring at least $u+1$ times and the partitions of $m$ with each part at least $u+1$. Therefore, $\sum_{i+uj=m} b_{i,j}=a_m.$
\end{proof}

\begin{lemma}
For each $n\geq 0$, let us define 
$$f_n(q):=\sum_{N=0}^\infty b_{N,n}q^N,$$
which is a formal power series in $q$. Then, 
$$\prod_{j=u+1}^\infty (1-q^j)^{-1}=\sum_{n=0}^\infty f_n(q)q^{nu}.$$
\end{lemma}

\begin{proof}
Note that,
$$\prod_{j=u+1}^\infty (1-q^j)^{-1}=\sum_{m=0}^\infty a_mq^m $$
since, for each $m$, the coefficient of $q^m$ on LHS is equal to the number of partitions of $m$ with each part at least $u+1$. Then,

\begin{flalign*}
\sum_{n=0}^\infty f_n(q)q^{nu}& =\sum_{n=0}^\infty\left(\sum_{N=0}^\infty b_{N,n}q^N\right)q^{nu}\\
& =\sum_{n=0}^\infty \sum_{N=0}^\infty b_{N,n}q^{N+nu}\\
& =\sum_{m=0}^\infty \left(\sum_{i+uj=m}b_{i,j}\right)q^m\\
& =\sum_{m=0}^\infty a_mq^m\\
& =\prod_{j=u+1}^\infty(1-q^j)^{-1}.
\end{flalign*}
\end{proof}

We also need the following lemma, which computes the cardinality of $\Aut (G)$, for a finite abelian $p$-group $G$.

\begin{lemma}
Fix a prime $p$. Suppose $G$ is a finite abelian $p$-group and $$G=\prod_{i=1}^k\left(\mathbb{Z}/p^{e_i}\mathbb{Z}\right)^{r_i}$$ for some $k\geq 0$, $e_1>e_2>\cdots >e_k>0$ and $r_i\geq 0.$ Then 
$$\left|\Aut (G)\right|=\left(\prod_{i=1}^k\left(\prod_{s=1}^{r_i}(1-p^{-s})\right)\right)\left( \prod_{1\leq i,j\leq k}p^{\min (e_i,e_j)r_ir_j}\right).$$
\end{lemma}

\begin{proof}
See [3], Theorem 1.2.10.
\end{proof}

Now let us return to the proof of the theorem. We follow a similar argument given in \cite{hall} or in the proof of Theorem 2.1.2 of \cite{lengler}. First, note that, there is an associated partition corresponding to every finite abelian $p$-group and corresponding to every partition there is an associated finite abelian $p$-group; this comes from writing finite abelian $p$-groups uniquely as a product of cyclic groups. For example, if we write a finite abelian $p$-group $G$ as, 
$$G=\prod_{i=1}^k\mathbb{Z}/p^{e_i}\mathbb{Z}$$ where $e_1\geq e_2\geq\cdots \geq e_k>0$, then the associated partition $\lambda$ is given by $\lambda =(e_1,e_2,\ldots ,e_k)$. And, corresponding to every partition $\lambda =(\lambda_1,\lambda_2,\ldots ,\lambda_k)$ the associated $p$-group $G_\lambda$ is given by $G_\lambda =\prod_{i=1}^k\mathbb{Z}/p^{\lambda_i}\mathbb{Z}.$ Note that, if $|\lambda |$ denotes the size of the partition $\lambda$, then the order of the $p$-group $G_\lambda$ is given by $\left| G_\lambda\right| =p^{|\lambda |}$.

Let $\lambda :=(\lambda_1,\ldots ,\lambda_l)$ be a partition of size $n$ and suppose $\lambda^\prime :=(\lambda^\prime_1,\ldots ,\lambda_m^\prime)$ is it's conjugate partition. Then, note that, in $G_\lambda$ (as a product of cyclic groups), the factor $\mathbb{Z}/p^i\mathbb{Z}$ occurs exactly $\lambda^\prime_i-\lambda^\prime_{i+1}$ times (where $\lambda_{m+1}^\prime :=0).$ Then using Lemma 3 we can write 
\begin{flalign*}
\left|\Aut (G_\lambda )\right|& =\left(\prod_{i=1}^m\left(\prod_{s=1}^{\lambda^\prime_i-\lambda^\prime_{i+1}}(1-p^{-s})\right)\right)\left(\prod_{1\leq i,j\leq m} p^{\min (i,j)(\lambda^\prime_i-\lambda^\prime_{i+1})(\lambda^\prime_j-\lambda^\prime_{j+1})}\right)\\
&= \left(\prod_{i=1}^m\left(\prod_{s=1}^{\lambda^\prime_i-\lambda^\prime_{i+1}}(1-p^{-s})\right)\right) p^{\sum_{1\leq i,j\leq m}\min (i,j)(\lambda^\prime_i-\lambda^\prime_{i+1})(\lambda^\prime_j-\lambda^\prime_{j+1})}\\
& =\left(\prod_{i=1}^m\left(\prod_{s=1}^{\lambda^\prime_i-\lambda^\prime_{i+1}}(1-p^{-s})\right)\right) p^{\sum_{i=1}^m(\lambda_i^\prime )^2}.
\end{flalign*}
Then, setting $q=p^{-1}$, we have
\begin{flalign*}
\sum_{G\in\mathcal{G}_p}\frac{1}{|G|^u\cdot\left|\Aut (G)\right|}& =\sum_{n=0}^\infty\sum_{\substack{G_\lambda\in\mathcal{G}_p\\ |\lambda |=n}}\frac{q^{nu}}{\left|\Aut (G_\lambda )\right|}\\
& =\sum_{n=0}^\infty q^{nu}\sum_{\substack{G_\lambda\in\mathcal{G}_p\\ |\lambda |=n}}\left(\prod_{i=1}^m\left(\prod_{s=1}^{\lambda^\prime_i-\lambda^\prime_{i+1}}(1-p^{-s})^{-1}\right)\right)\left(\prod_{i=1}^m p^{-(\lambda_i^\prime )^2}\right)\\
& =\sum_{n=0}^\infty q^{nu}\sum_{\substack{G_\lambda\in\mathcal{G}_p\\ |\lambda |=n}}\left(\prod_{i=1}^m\left(\prod_{s=1}^{\lambda^\prime_i-\lambda^\prime_{i+1}}(1-q^{s})^{-1}\right)\right)\left(\prod_{i=1}^m q^{(\lambda_i^\prime )^2}\right).
\end{flalign*}
Note that, $\lambda^\prime$ varies over all partitions as $\lambda$ varies over all partitions. Therefore putting $\mu =\lambda^\prime$, we get
\begin{flalign*}
\sum_{G\in\mathcal{G}_p}\frac{1}{|G|^u\cdot\left|\Aut (G)\right|}& =\sum_{n=0}^\infty q^{nu}\sum_{\substack{G_\mu\in\mathcal{G}_p\\ |\mu |=n}}\left(\prod_{i=1}^m\left(\prod_{s=1}^{\mu_i-\mu_{i+1}}(1-q^{s})^{-1}\right)\right)\left(\prod_{i=1}^m q^{\mu_i^2}\right)\\
&=\sum_{n=0}^\infty q^{nu} \sum_{\substack{G_\mu\in\mathcal{G}_p\\ |\mu |=n}}\left(\prod_{i=1}^m\psi_{\mu_i,\mu_{i-1}-\mu_i}(q)\right)\left(\prod_{i=1}^m q^{\mu_i^2}\right) ,
\end{flalign*}
where $$\psi_{a,b}(q):=\frac{\prod_{i=1}^{a+b}(1-q^i)}{\prod_{i=1}^a(1-q^i)\prod_{i=1}^b (1-q^i)},\quad \psi_{a,\infty }(q):=\frac{1}{\prod_{i=1}^a (1-q^i)}$$ and $\mu_0:=\infty$; note that, the coefficient of $q^n$ in $\psi_{a,b}(q)$ is the number of partitions of $n$ with hight at most $a$ and width at most $b$.Therefore, by Lemma 2,  it is enough to show that, for each $n\geq 0$,
$$f_n(q)=\sum_{\substack{G_\mu\in\mathcal{G}_p\\ |\mu |=n}}\left(\prod_{i=1}^m\psi_{\mu_i,\,\mu_{i-1}-\mu_i}(q)\right)\left(\prod_{i=1}^m q^{\mu_i^2}\right).$$
That is, we need to equate coefficients of $q^N$ on both sides, for each $N\geq 0$. The argument is same as given in [2] or [3].

Note that, coefficient of $q^N$ on LHS is equal to $b_{N,n}$ which is the number of partitions of $N$ with greatest part $n$. Let $\nu$ be a partition of $N$ with greatest part equal to $n$; then, to each such $\nu$ we will associate a partition $\mu$ of size $n$ on the RHS. Consider the conjugate $\nu^\prime$ of $\nu$ and let $D$ be the standard \textit{Young diagram} of $\nu^\prime.$ Note that $\nu^\prime$ has hight equal to $n$. Now, define $\mu:=(\mu_1,\ldots ,\mu_m)$ as follows:
\begin{itemize}
\item
Define $\mu_1$ to be the largest integer such that $(\mu_1,\mu_1)\in D.$
\item
For $i\geq 2$, define $\mu_i$ to be the largest integer such that $$(\mu_1+\cdots +\mu_i,\mu_i)\in D.$$
\end{itemize}
(where $(i,j)\in D$ is defined as the block of $D$ situated at the $i$th row from top and $j$th column from left). Then $|\mu |=n$. If $M$ is the number of blocks outside the squares of size $\mu_i$ then $M=N-\mu_1^2-\mu_2^2-\cdots -\mu_m^2.$ Let $M_i$ be the number of blocks at the right of the block of size $\mu_i$, i.e. 
$$M_i:=\left|\{(x,y)\in D : \mu_1+\cdots +\mu_{i-1}<x<\mu_1+\cdots +\mu_i, \;\mu_i<y\}\right|.$$
Then the blocks corresponding to $M_i$ gives  a partition  of $M_i$ of hight at most $\mu_i$ and width at most $\mu_{i-1}-\mu_i$ and hence this contributes to the coefficient of $q^{M_i}$ in $\psi_{\mu_i,\mu_{i-1}-\mu_i}(q)$ on RHS. Note that $M=M_1+\cdots +M_m$ which implies $M_1+\cdots +M_m+\mu_1^2+\cdots +\mu_m^2=N$ and hence $\mu$ contributes to the coefficient of $q^N$ on RHS.

Note that, the above construction can be reversed. Suppose $\mu$ is a partition which corresponds to the coefficient of $q^N$ on RHS such that $\mu$ is specified by the numbers $M_i$, where $M_1+\cdots +M_m+\mu_1^2+\cdots +\mu_m^2=N$, and partitions of $M_i$ of hight at most $\mu_i$ and width at most $\mu_{i-1}-\mu_i.$ Then we can construct the Young diagram $D$ and construct the corresponding partition $\nu$ on LHS. Hence, we conclude that the coefficients of $q^N$ on both sides are equal and this proves the theorem.

\section*{Author information}
Pritam Majumder, Department of Mathematics and Statistics,

IIT Kanpur, Kanpur-208016, U.P., India.

Email: \texttt{pritamaj@gmail.com}

\end{document}